\documentclass[12pt]{amsart}

\newtheorem{theorem}{Theorem}[section]
\newtheorem{lemma}{Lemma}[section]

\newtheorem{prop}{Proposition}[section]
\newtheorem{remark}{Remark}[section]
\newtheorem{defi}{Definition}[section]

\def\O{\Omega}
\def\R{{\mathbb R}}

\def\g{\gamma}

\begin{document}
\section*{}

\setcounter{equation}{0}
\title[]{Improved Poincar\'e inequalities in fractional Sobolev spaces}

\author{Irene Drelichman}
\address{IMAS (UBA-CONICET) \\
Facultad de Ciencias Exactas y Naturales\\
Universidad de Buenos Aires\\
Ciudad Universitaria\\
1428 Buenos Aires\\
Argentina}
\email{irene@drelichman.com}

\author{Ricardo G. Dur\'an}
\address{IMAS (UBA-CONICET) and Departamento de Matem\'atica\\
Facultad de Ciencias Exactas y Naturales\\
Universidad de Buenos Aires\\
Ciudad Universitaria\\
1428 Buenos Aires\\
Argentina}
\email{rduran@dm.uba.ar}


\thanks{Supported by ANPCyT under grant PICT 2014-1771, and by Universidad de Buenos Aires under grant 20020120100050BA. The authors are members of
CONICET, Argentina.}


\begin{abstract}
We obtain improved fractional Poincar\'e and Sobolev Poincar\'e inequalities including powers of the distance to the boundary in John, $s$-John domains and H\"older-$\alpha$ domains, and discuss their optimality.

\end{abstract}
\maketitle

\section{Introduction}
\label{intro}

Poincar\'e and Sobolev-Poincar\'e inequalities in non-Lipschitz domains have been the object of extensive study.  They can be seen as special cases of the following larger family of so-called \emph{improved Poincar\'e inequalities}:
\begin{equation}
\label{poinc-intro}
\inf_{c\in\R} \|f(x) - c\|_{L^q(\O)} \le C  \|\nabla f(x) \, d(x)^b \|_{L^p(\O)}
\end{equation}
where $d(x)$ denotes the distance to the boundary of $\O$, $b\in [0,1]$, and
$p$ and $q$ satisfy appropriate restrictions.
The usual assumption for these inequalities to hold is that the domain $\O \subset \R^n$
belongs to the class of John or $s$-John domains (they will be called $\beta$-John in this paper,
since $s$ is reserved for the fractional derivative), see Section 2 for a precise definition.
In the case of John domains, a partial converse is also true in the following sense:
if $\O$ has finite measure and satisfies a separation property, then the validity of the
Sobolev-Poincar\'e inequality implies the John condition (see \cite{BK}).
A possibly incomplete list of references on improved Poincar\'e inequalities and their
generalizations to weighted settings and measure spaces includes
\cite{C, CW1, CW2, DD, Ha, HK,  H, H2, KM, M}.

More recently, some authors have turned their attention to fractional generalizations
of Poincar\'e and Sobolev-Poincar\'e inequalities, where a fractional seminorm appears instead of the  norm in $W^{1,p}(\O)$.
Indeed, in \cite{D, HV} the following inequalities were introduced for John domains:
\begin{equation}
\label{spoincf2-intro}
\inf_{c\in\R}\|f(y) - c\|_{L^q(\O)}
\le C \left\{
\int_\O\int_{\O \cap B^n(x,\tau dist(x,\partial \O))}\frac{|f(z)-f(x)|^p}{|z-x|^{n+sp}} \, dz dx
\right\}^{1/p}
\end{equation}
with $1\le p\le q\le \frac{np}{n-sp}$ and $s, \tau \in (0,1)$.

The seminorm appearing on the RHS of \eqref{spoincf2-intro} can be seen to be equivalent on Lipschitz domains to the usual seminorm in $W^{s,p}(\O)$, that is, integrating over $\O\times\O$ (see \cite[equation (13)]{Dy}), but it can be strictly smaller than the usual seminorm for general John domains (see \cite[Proposition 3.4]{D}). Moreover, it is easy to see that, unlike the classical Poincar\'e inequality, the inequality
\begin{equation}
\label{poincf-intro}
\inf_{c\in\R} \|f(y) - c\|_{L^p(\O)} \le C   \left\{ \int_\O \int_\O \frac{|f(z)-f(x)|^p}{|z-x|^{n+sp}} \, dy \, dx \right\}^{1/p}
\end{equation}
holds for any bounded domain $\O$ and $s\in (0,1)$ (see Section 2), while the stronger inequality
\begin{equation}
\label{poincfd-intro}
\inf_{c\in\R} \|f(y) - c\|_{L^p(\O)} \le C   \left\{ \int_\O \int_{\O \cap B^n(x,\tau dist(x,\partial \O))} \frac{|f(z)-f(x)|^p}{|z-x|^{n+sp}} \, dy \, dx \right\}^{1/p}
\end{equation}
may fail for general domains, for instance, for certain $\beta$-John domains if  $\beta$ is sufficiently large (see Theorem \ref{mushroom}).

Regarding the Sobolev-Poincar\'e inequality,  it was proved in \cite[Theorem 1.2]{Z} that
\begin{equation}
\label{spoincf-intro}
\inf_{c\in\R} \|f(y) - c\|_{L^\frac{np}{n-sp}(\O)} \le C   \left\{\int_\O \int_\O \frac{|f(z)-f(x)|^p}{|z-x|^{n+sp}} \, dy \, dx \right\}^{1/p}
\end{equation}
holds for the class of Ahlfors $n$-regular domains, which is larger than that of the John domains,
 but if we turn to the inequality with the stronger seminorm, there are Ahlfors $n$-regular domains for which the
 inequality fails (see \cite[Theorem 3.1]{D}).
  On John domains, as mentioned before, the Sobolev-Poincar\'e inequality holds with the stronger seminorm
  and, moreover, it was proved in \cite[Theorem 6.1]{D}
  that a partial converse also holds:
if a fractional Sobolev-Poincar\'e inequality with the stronger seminorm holds on a
domain $\O$ with finite measure which satisfies the separation property, then $\O$ must
satisfy the John condition.

In this paper, we study  generalizations of \eqref{spoincf2-intro} which include -on both sides- weights that are a power of the distance to the boundary. More precisely, we  obtain improved inequalities of the form
\begin{equation}
\label{frac-intro}
\inf_{c\in\R}\|f(y) - c\|_{L^q(\O, d^a)}
\le C \left\{
\int_\O\int_{|x-z|\le d(x)/2}\frac{|f(z)-f(x)|^p}{|z-x|^{n+sp}} \, \delta(x,z)^{b} dz dx
\right\}^{1/p}
\end{equation}
where $d(x):=dist(x,\partial\O)$,  $\delta(x,z):=\min\{d(x),d(z)\}$, $\Omega\subset \R^n$ is a John or $\beta$-John domain and the parameters satisfy appropriate restrictions. The reader will remark that  the  domain of integration on the left corresponds to the choice $\tau=\frac12$ in the notation of \eqref{spoincf2-intro}; this is to simplify notation, we could have chosen any $\tau\in(0,1)$ as it will be clear from the proof. We also remark that the term ``improved'' used in \cite{D} refers to the use of the stronger seminorm as in \eqref{poincfd-intro}, while in this paper we use it to emphasize the presence of powers of the distance to the boundary as weights, as it is customary in the integer case.

Our technique consists in extending the arguments used in our work \cite{DD} to the fractional case. The key starting point in that paper was the estimate
$$
|f(y)-\bar f|
\le C\int_{|x-y|\le C_1 d(x)}\frac{|\nabla f(x)|}{|x-y|^{n-1}}\,dx
$$
where $\O$ is a John domain, $f\in C^\infty (\O)$ and $\bar f$ is an appropriate average of $f$.
The idea of recovering $f$ from its gradient to prove Sobolev-Poincar\'e inequalities is present
in several authors, for instance, \cite{M, Ha, H}, but it is essential for our argument in \cite{DD}
that the fractional integral of the gradient be restricted to the region $|x-y|<C_1 d(x)$,
a fact that we believe is not exploited in other proofs. In this paper we give a generalization
of this representation to the fractional case in the case of John and $\beta$-John domains,
that can also be of independent interest. We also consider separately the case of
H\"older-$\alpha$ domains, which belong to the class of $\beta$-John domains with $\beta=1/\alpha$
but are known to have better embedding properties, see e.g. \cite{BS, KM}.

To the best of our knowledge, the fractional inequalities for $\beta$-John domains are new even in the unweighted case, and the weighted inequalities are new even in the case of Lipschitz domains. Indeed, although the generalization to weighted norms on both sides of the inequality is quite natural and along the lines of the results for improved Poincar\'e inequalities involving the gradient found  in \cite{CW1, CW2, HK, KM},  we believe that the only antecedent of these weighted fractional inequalities is found in \cite[Proposition 4.7]{AB}, where \eqref{frac-intro} is obtained in a star-shaped domain in the case $p=q=2$, $a=0$ and $b<2s$ (their proof remains unchanged for John domains but  does not cover the case $b=2s$ where the inequality also holds, see \cite[Remark 4.8]{AB} and Theorem \ref{teorema en John} below). Moreover, the results we obtain are sharp in the case of John domains and H\"older-$\alpha$ domains, and almost sharp (except at the endpoint) for $\beta$-John domains, and we provide counterexamples to support this statement.

The rest of the paper is as follows: in Section 2 we recall some necessary definitions and preliminaries; in Section 3 we obtain the fractional representation in John domains and use it to obtain the improved inequalities; in Section 4 we obtain the fractional representation in $\beta$-John domains and use it to obtain the improved inequalities, and we discuss their optimality; finally, in Section 5 we consider the special case of H\"older-$\alpha$ domains, also discussing the optimality of our result.

\section{Notation and preliminaries}

Throughout the paper, $\O\subset \R^n$ ($n\ge 2$) will be a bounded domain and $d(x)$ will denote the distance of a point $x\in\O$ to the boundary of $\O$. We will assume,  without loss of generality, that $0\in\O$.

As it is customary, $C$  will denote a positive constant that may change even within a single string of inequalities, and functions $f$ defined in $\O$ will be extended by zero outside $\O$ whenever needed.

For completeness, we include the following elementary result mentioned in the introduction:
\begin{prop}
The fractional Poincar\'e inequality
$$
\inf_{c\in \R} \|f-c\|_{L^p(\O)} \le \left\{ \frac{diam(\O)^{n+sp}}{|\O|} \int_\O \int_\O \frac{|f(y)-f(x)|^p}{|y-x|^{n+sp}} \, dy \, dx \right\}^{1/p}
$$
holds for any bounded  domain $\O\subset \R^n$.
\end{prop}
\begin{proof}
Let $f_\O= \frac{1}{|\O|}\int_\O f(x) \, dx$. Then, by Minkowski's integral inequality,
\begin{align*}
\|f-f_\O\|_{L^p(\O)}&=\left\|\frac{1}{|\O|} \int_\O (f(y) - f(x)) \, dx\right\|_{L^p} \\
&\le \frac{1}{|\O|} \int_\O \Big( \int_\O |f(y)-f(x)|^p \, dy \Big)^\frac{1}{p} \, dx
\end{align*}
Hence, by H\"older's inequality,
\begin{align*}
\|f-f_\O\|_{L^p(\O)}^p &\le \frac{1}{|\O|} \int_\O \int_\O |f(y)-f(x)|^p \, dy \, dx \\
&\le  \frac{diam(\O)^{n+sp}}{|\O|} \int_\O \int_\O \frac{|f(y)-f(x)|^p}{|y-x|^{n+sp}} \, dy \, dx
\end{align*}
\end{proof}
We remark that the constant in the previous inequality is  far from being sharp, see for instance \cite[Theorem 1]{BBM} for the best constant when $\O$ is a cube.

We will use the following definition of $\beta$-John domains:

\begin{defi}
A bounded domain $\O\subset\R^n$ is a $\beta$-John domain ($\beta\ge 1$) if there exists a family of
rectifiable curves given by $\g(t,y)$, $0\le t\le 1$, $y\in\O$,
and positive constants $\lambda$, $K$ and $C$ such that,

\begin{enumerate}
\item $\gamma(0,y)=y$, $\gamma(1,y)=0$
\medskip
\item  $d(\gamma(t,y))\ge \lambda t^\beta$ \label{prop2}
\medskip
\item  $|\dot\gamma(t,y))|\le K$
\medskip
\item  for all $x,y \in \Omega$ and $r\le \frac12 d(x)$, there holds $l(\gamma(y)\cap B(x,r))\le Cr
$, where $\gamma(y)$ denotes the curve joining $0$ with $y$, and $l$ the length. \label{prop5}
\end{enumerate}
When $\beta=1$, we will simply refer to John domains, instead of 1-John domains.
\end{defi}

\begin{remark}
The above definition is not the usual one, which includes only properties (1), (2) and (3).
However, it can be seen that the curves can be chosen to make property (4) hold (see \cite[Section 2]{DMRT}).
\end{remark}

\section{The case of John domains}

In this section we obtain a representation that allows us to estimate $f$ in terms of its fractional derivative, and use it to obtain the inequalities in John domains. These are split in two theorems, since the case $p=1$ requires a ``weak implies strong'' argument that we develop separately. The inequalities are sharp
(as will be seen in Theorem \ref{mushroom}) and, to the best of our knowledge, they are new even in the case of Lipschitz domains.

\begin{prop}
\label{representacion fraccionaria}
Given $s\in(0,1)$, $1\le p<\infty$, and $f\in C^\infty(\O)$ we have
$$
|f(y) - \bar f|
\le C \int_{|y-x|\le C_1d(x)} \frac{h(x)}{|x-y|^{n-s}}\, dx
$$
where $\bar f$ is a constant and
$$
h(x)
:=\left(\int_{|x-w|\le d(x)/2} \frac{|f(w)-f(x)|^p}{|w-x|^{n+sp}} \, dw\right)^\frac{1}{p}
$$
for $y\in\O$, $h\equiv 0$ outside $\O$, and $C$ and $C_1$ are positive
constants depending only on $n$ and $\O$.
\end{prop}
\begin{proof}
Take $\varphi\in C^1_0(B(0, \lambda/2))$
such that $\int\varphi=1$ and define
$$
u(x,t) =(f * \varphi_t)(x)
$$
and
$$
\eta(t)=u(\gamma(t,y)+tz , t).
$$
Observe that the curve $\g(t,y)+tz$ is contained in $\O$ whenever $z\in B(0,\lambda/2)$.
Indeed, in this case
$|\g(t,y)+tz-\g(t,y)|\le t\lambda/2 < d(\g(t,y))$.

Then,
\begin{align*}
&f(y) - (f*\varphi)(z) = u(y,0) - u(z,1)
= \eta(0) - \eta(1) = - \int_0^1 \eta'(t)\, dt \\
&= - \int_0^1 \nabla u(\g(t,y)+tz , t) \cdot (\dot \g (t, z) + z)
+ u_t (\g(t,y) + tz, t) \, dt
\end{align*}

Multiplying by $\varphi(z)$, integrating in $z$ and defining $\bar f=\int (f*\varphi)(z)\varphi(z)dz$
we have

\begin{align*}
f(y) - \bar f &= \int_{\R^n} (f(y) - (f*\varphi)(z)) \varphi(z)  \, dz \\
&= - \int_{\R^n} \int_0^1 \nabla u(\gamma(t,y)+tz , t) \cdot (\dot \g(t,y) + z) \varphi(z) \, dt dz \\
& \quad - \int_{\R^n} \int_0^1 \frac{\partial u}{\partial t}(\g(t,y) + tz, t)  \varphi(z) \, dt dz \\
&= - (I + II)
\end{align*}

Making the change of variables $\g(t,y) + tz = x$ and using that
$$
\nabla u= f * \nabla(\varphi_t)
$$
and
$$
\nabla(\varphi_t)(x)
=\frac{1}{t^{n+1}} \nabla\varphi\Big(\frac{x}{t}\Big)
$$
we obtain

\begin{align*}
I &= \int_0^1 \int_{\R^n}
\int_{\R^n}  \frac{f(w)}{t^{n+1}}
\nabla\varphi\Big(\frac{x-w}{t}\Big) \, dw
\cdot
\left( \dot \g(t,y) + \frac{x-\g(t,y)}{t}\right) \varphi\Big(\frac{x-\g(t,y)}{t}\Big) \, dx
\frac{dt}{t^n}
\end{align*}

Observe that, since the support of $\varphi$ is contained in $B(0,\lambda/2)$, the integrand
vanishes unless $|x-\gamma(t,y)|\le \lambda t/2$ which implies
\begin{equation}
\label{cota x menos y 1}
|x-y|\le |x-\gamma(t,y)| + |\gamma(t,y)-y|\le \frac{\lambda t}2 + \sqrt{n}Kt .
\end{equation}
Then we can restrict the integral to $t>c|x-y|$ with a constant $c$ depending
only on $K$, $\lambda$ and $n$.

On the other hand, using that
$$
\int \frac{1}{t^{n+1}} \nabla\varphi\Big(\frac{x-w}{t}\Big) \, dw =0
$$
we can subtract $f(x)$ in the integral with respect to $w$. Then,
changing the order of integration between $t$ and $x$
and using that
$$
\left|\dot \g(t,y) + \frac{x-\gamma(t,y)}{t} \right|\le K+\frac{\lambda}2 ,
$$
we obtain

$$
I \le  C  \int_{\R^n} \int_{c|x-y|}^1
\int_{|x-w|\le\lambda t/2}
\frac{|f(w)-f(x)|}{t^{n+1}} \left|\nabla\varphi \Big(\frac{x-w}{t}\Big)\right| \left|\varphi \Big(\frac{x-\gamma(t,y)}{t}\Big)\right| \, dw
\,  \frac{dt}{t^n} dx
$$
with a constant $C$ depending only on $K$ and $\lambda$.

Now observe that
$$
d(\gamma(t,y))\le |\gamma(t,y)-x| + d(x)  \le \frac{\lambda t}{2} + d(x)
\le \frac{d(\gamma(t,y))}{2} + d(x)
$$
and so
$$
|x-w| \le \frac{\lambda t}{2} \le \frac{d(\gamma(t,y))}{2} \le \frac{d(x)}2.
$$
In particular $\lambda t/2\le d(x)/2$ which combined with \eqref{cota x menos y 1} gives
$$
|x-y|\le C_1 d(x)
$$
with a constant $C_1$ depending only on $K$, $\lambda$  and $\|\varphi\|_\infty$.
Consequently,
\begin{align*}
I &\le  C \int_{|x-y|\le C_1 d(x)} \int_{c|x-y|}^1
\int_{|x-w|\le d(x)/2} \frac{|f(w)-f(x)|}{|w-x|^{\frac{n}{p}+s}}
\frac{1}{t^{n+\frac{n}{p'}+1-s}}\left|\nabla\varphi \Big(\frac{x-w}{t}\Big)\right|
\, dw \, dt \,dx \\
&\le C \int_{|x-y|\le C_1 d(x)} \int_{c|x-y|}^1
\left( \int_{|x-w|\le d(x)/2} \frac{|f(w)-f(x)|^p}{|w-x|^{n+sp}} \, dw
\right)^\frac{1}{p} \left( \int_{\R^n} \left|\nabla\varphi
\Big(\frac{x-w}{t}\Big)\right|^{p'} \, dw  \right)^\frac{1}{p'}\\
&\quad\cdot \frac{1}{t^{n+\frac{n}{p'}+1-s}} \, dt \,dx
\end{align*}
where we have used $|x-w|\le \lambda t/2$ to bound the integrand.

Therefore, since
$$
\left( \int_{\R^n} \left|\nabla\varphi
\Big(\frac{x-w}{t}\Big)\right|^{p'} \, dw  \right)^\frac{1}{p'}
=\|\nabla\varphi\|_{p'} t^{\frac{n}{p'}}
$$
we conclude that
$$
I \le  C  \int_{|x-y|\le C_1 d(x)} \int_{c|x-y|}^1 h(x) \frac{1}{t^{n+1-s}} \, dt \,dx
\le  C  \int_{|x-y|\le C_1 d(x)} \frac{h(x)}{|x-y|^{n-s}}   \,dx
$$
where the new constant depends also on $\|\nabla\varphi\|_{p'}$.

To estimate $II$ we proceed in a similar way. Indeed, since $\int\varphi_t(x)dx=1$ for all
$t$, we have $\int\frac{\partial\varphi_t}{\partial t}(x)dx=0$. Moreover, a straightforward
computation shows that
$$
\frac{\partial\varphi_t}{\partial t}(x)
=\frac1{t^{n+1}}\psi\Big(\frac{x}{t}\Big)
$$
where $\psi :=-n\varphi-x\cdot\nabla\varphi$. Therefore, repeating the arguments
that we used to bound $I$ we obtain,
$$
II \le  C  \int_{|x-y|\le C_1 d(x)} \int_{c|x-y|}^1
\int_{|x-w|\le d(x)/2}
\frac{|f(w)-f(x)|}{t^{n+1}} \left|\psi\Big(\frac{x-w}{t}\Big)\right| \, dw
\,  \frac{dt}{t^n} dx
$$
and consequently
$$
II \le  C  \int_{|x-y|\le C_1 d(x)} \frac{h(x)}{|x-y|^{n-s}}   \,dx
$$
\end{proof}

In the proof of the next Theorem we will make use of the following well known result
(see, e.g., \cite[Lemma 2.8.3]{Zi}).

\begin{lemma}
\label{maximal}
Let $Mg$ be the Hardy-Littlewood maximal function of $g$. Given $0<\sigma<n$
there exists a positive constant $C$, depending only on $n$ and $\sigma$, such that, for any $\varepsilon>0$,
$$
\int_{|y-x|\le\varepsilon} \frac{|g(y)|}{|x-y|^{n-\sigma}} dy
\le C \varepsilon^{\sigma} Mg(x)
$$
\end{lemma}

\begin{theorem}
\label{teorema en John}
Let $\Omega\subset\R^n$ be a John domain,  $1< p\le q<\infty$,  $a\ge 0$, $b\le \frac{(n+a)p}{q}+sp-n$ and,
additionally, $q\le \frac{np}{n-sp}$ when $p<\frac{n}{s}$. Then, given $s\in(0,1)$ and $f\in C^\infty(\O)$ we have
$$
\inf_{c\in \R}\|f(y) - c\|_{L^q(\O, d^a)}
\le C \left[
\int_\O\int_{|x-z|\le d(x)/2}\frac{|f(z)-f(x)|^p}{|z-x|^{n+sp}}\delta(x,z)^b dz dx\right]^{1/p}
$$
where $\delta(x,z):=\min\{d(x),d(z)\}$.
\end{theorem}

\begin{proof}
We proceed by duality. Let $g\in L^{q'}(\Omega, d^a)$ such that  $\|g\|_{L^{q'}(\O, d^a)}=\|g d^\frac{a}{q'}\|_{L^{q'}(\O)}=1$. Interchanging the order of integration
and using Proposition \ref{representacion fraccionaria}
we have
\begin{equation}
\label{dualidad}
\begin{aligned}
\int_\O |f(y) - \bar f| g(y) d(y)^a dy
&\le C \int_\O \int_{|y-x|\le C_1d(x)} \frac{|g(y)|}{|x-y|^{n-s}} d(y)^{\frac{a}{q}+\frac{a}{q'}} dy \,h(x) dx\\
&\le C \int_\O \int_{|y-x|\le C_1d(x)} \frac{|g(y)|}{|x-y|^{n-s}} d(y)^\frac{a}{q'} dy \,h(x) d(x)^\frac{a}{q} dx\\
\end{aligned}
\end{equation}
where we have used that $d(y)\le |x-y|+d(x) \le (C_1+1) d(x)$ in the region of integration.

Now we consider separately the cases $p=q$ and $p<q$.

If $p=q$, it is clear that it suffices to prove the statement for $b=a+sp$. Using Lemma \ref{maximal} we have that
\begin{align}
\int_\O |f(y) - \bar f| g(y) d(y)^a dy
&\le C \int_\O d(x)^{s+\frac{a}{p}} M(g d^\frac{a}{p'})(x) h(x)  dx \nonumber\\
&\le C \|d^{s+\frac{a}{p}}h\|_{L^p(\O)}\|M(g d^\frac{a}{p'})\|_{L^{p'}(\O)}\label{cota maximal en John}
\end{align}

But,
$$
h(x)^p
=\int_{|x-z|\le d(x)/2} \frac{|f(z)-f(x)|^p}{|z-x|^{n+sp}} \, dz
$$
and then,
$$
\|d^{s+\frac{a}{p}}h\|^p_{L^p(\O)}
=\int_\O d(x)^{sp+a}
\int_{|x-z|\le d(x)/2}\frac{|f(z)-f(x)|^p}{|z-x|^{n+sp}} \, dz dx
$$
but in the domain of integration $d(x)\le 2d(z)$ and, therefore,
$$
\|d^{s+\frac{a}{p}}h\|^p_{L^p(\O)}
\le C\int_\O \int_{|x-z|\le d(x)/2}\delta(x,z)^{sp+a}
\frac{|f(z)-f(x)|^p}{|z-x|^{n+sp}} \, dz dx.
$$
Replacing this estimate in \eqref{cota maximal en John}
 we conclude the proof using the boundedness of the
maximal operator in $L^{p'}$ and the choice of $g$.

If $p<q$,  assume first that $\frac{(n+a)p}{n+b-sp}\le \frac{np}{n-sp}$. Then, for fixed $p,a,b$ it suffices to prove the theorem for $q=\frac{(n+a)p}{n+b-sp}$.
 If we define $\eta=\frac{b}{p}-\frac{a}{q}$ it follows from our assumptions that $0\le \eta <s$
  (the first inequality using that $\frac{(n+a)p}{n+b-sp}\le \frac{np}{n-sp}$ and the second one using that $p<\frac{(n+a)p}{n+b-sp}$). Therefore,  by  \eqref{dualidad} and using that $|x-y|\le C_1 d(x)$ we have
\begin{align}
\int_\O |f(y) - \bar f| g(y) d(y)^a dy
&\le C \int_\O d(x)^{\frac{b}{p}} I_{s-\eta}(g d^\frac{a}{q'})(x) h(x)  dx \nonumber\\
&\le C \|d^{\frac{b}{p}}h\|_{L^p(\O)}\|I_{s-\eta}(g d^\frac{a}{q'})\|_{L^{p'}(\O)}\label{cota int fraccionaria en John}
\end{align}
where $I_{\gamma}g(x)=\int \frac{g(y)}{|x-y|^{n-\gamma}} dy$ is the fractional integral
(or Riesz potential) of order $\gamma$ of $g$, provided $0<\gamma<n$.
Observe that, indeed, $0<s-\eta<n$ holds because $0\le \eta<s$ and $s\in (0,1)$.

As before,
\begin{align*}
 \|d^{\eta+\frac{a}{q}}h\|^p_{L^p(\O)}
&\le C\int_\O \int_{|x-z|\le d(x)/2}\delta(x,z)^{\eta p+\frac{ap}{q}}
\frac{|f(z)-f(x)|^p}{|z-x|^{n+sp}} \, dz dx\\
&= C\int_\O \int_{|x-z|\le d(x)/2}\delta(x,z)^b
\frac{|f(z)-f(x)|^p}{|z-x|^{n+sp}} \, dz dx.
\end{align*}

Using this estimate in \eqref{cota int fraccionaria en John} we conclude the proof using the boundedness of the fractional integral $I_{s-\eta}: L^{q'} \to L^{p'}$ for $\frac{1}{q}=\frac{1}{p}-\frac{s-\eta}{n}$ and the choice of $g.$

It remains to consider the case $p<q$,  $\frac{(n+a)p}{n+b-sp}> \frac{np}{n-sp}$. In this case, for fixed $p,a,b$ it suffices to consider $q=\frac{np}{n-sp}$. Then, we may bound
\begin{align*}
\int_\O |f(y) - \bar f| g(y) d(y)^a dy
&\le C \int_\O  I_{s}(g d^\frac{a}{q'})(x) h(x) d(x)^{\frac{a}{q}} dx \\
&\le C \|d^{\frac{a}{q}}h\|_{L^p(\O)}\|I_{s}(g d^\frac{a}{q'})\|_{L^{p'}(\O)}
\end{align*}
and conclude by using the boundedness of $I_{s}: L^{q'} \to L^{p'}$ for $\frac{1}{q}=\frac{1}{p}-\frac{s}{n}$ and the fact that  $\|d^{\frac{a}{q}}h\|_{L^p(\O)}\le C  \|d^{\frac{b}{p}}h\|_{L^p(\O)}$ because, under our assumptions, $\frac{b}{p}\le \frac{a}{q}$.
\end{proof}

For the case $p=1$ we will make use of the following ``weak implies strong" result. It is proved in \cite[Theorem 4.1]{D} in the case $\mu=\nu$, but the reader can easily check that the same proof holds for two different measures.

\begin{lemma}\label{weak-strong}
Let $\mu$ and $\nu$ be positive Borel measures on an open set $\Omega \subset \R^n$, such that $\mu(\Omega)<\infty$, $\nu(\Omega)<\infty$, let $0<s<1$ and $1\le p\le q<\infty$. Then the following conditions are equivalent:
\begin{enumerate}
\item There is a constant $C_1>0$ such that the inequality
$$
\inf_{c\in\R} \sup_{t>0} \mu(\{ x\in \Omega: |f(x)-c|>t\}) t^q  \le C_1 \Big(\int_\Omega
\int_{|x-y|<d(x)/2} \frac{|f(y)-f(z)|^p}{|y-z|^{n+sp}} d\nu(z) d\nu(y) \Big)^\frac{q}{p}
$$
for any $f\in C^\infty(\O)$.

\item There is a constant $C_2>0$ such that the inequality
$$
\inf_{c\in\R} \int_\Omega |f(x)-c|^q \, d\mu(x) \le C_2 \Big(\int_\Omega
\int_{|x-y|<d(x)/2} \frac{|f(y)-f(z)|^p}{|y-z|^{n+sp}} d\nu(z) d\nu(y) \Big)^\frac{q}{p}
$$
holds, for every $f\in C^\infty(\O)$.
\end{enumerate}
\end{lemma}

\begin{theorem}
\label{caso 1 en John}
Let $\Omega\subset\R^n$ be a John domain, $1\le q\le \frac{n}{n-s}$, $a\ge 0$,   and
 $b\le\frac{(n+a)}{q}-n+s$. Then, given $s\in(0,1)$ and $f\in C^\infty(\O)$ we have
$$
\inf_{c\in\R} \|f(y) - c\|_{L^q(\O, d^a)}
\le C
\int_\O\int_{|x-z|\le d(x)/2}\frac{|f(z)-f(x)|}{|z-x|^{n+s}}\delta(x,z)^b dz dx
$$
where $\delta(x,z):=\min\{d(x),d(z)\}$.
\end{theorem}
 \begin{proof}

 If $q=1$ the result follows as in the previous proof.

 If $q>1$, it is clear that it suffices to prove our statement for $b=\frac{(n+a)}{q}-n+s$.
 For this purpose, we will prove a weak  inequality first. Hence, we  let $E=\{y\in \Omega : |f(y)-\bar f|>t\}$ and consider the measure $\mu$ such that $d\mu(x)=d(x)^a \, dx$.

Then,
 \begin{align*}
 \mu(E)&\le C \int_E \int_{|x-y|<C_1 d(x)} \frac{h(x)}{t |x-y|^{n-s}} \, dx  \,  d(y)^a \, dy\\
 &\le C \int_\O \frac{h(x)}{t} \int_{E \cap B(x, C_1d(x))} \frac{d(y)^a}{|x-y|^{n-s}} \, dy  \, dx\\
& = I_1 + I_2
 \end{align*}
 where $I_1$ corresponds to the region where $|x-y|<\frac{d(x)}{2}$ and $I_2$ to its complement.

 Observe that when $|x-y|<\frac{d(x)}{2}$, we have that $\frac{d(x)}{2}\le d(y)\le \frac32 d(x)$, so that
\begin{align*}
I_1 &\le C \int_{|x-y|< d(x)/2} \frac{h(x)}{t} \int_{E \cap B(x, C_1d(x))} \frac{1}{|x-y|^{n-s}} \, dy  \, d(x)^a \, dx\\
&\le C \int_{|x-y|< d(x)/2} \frac{h(x)}{t} |E \cap B(x, C_1d(x))|^\frac{s}{n}  \, d(x)^a \, dx\\
&\le C \int_{|x-y|< d(x)/2} \frac{h(x)}{t} \Big(\int_{E \cap B(x, C_1d(x))} \chi(y) d(y)^a \, dy \Big)^\frac{s}{n} \, d(x)^{a(1-\frac{s}{n})} \, dx\\
&= C \int_{|x-y|< d(x)/2} \frac{h(x)}{t} \mu(E \cap B(x, C_1d(x)))^\frac{s}{n} \, d(x)^{a(1-\frac{s}{n})} \, dx\\
&\le C \int_{|x-y|< d(x)/2} \frac{h(x)}{t} \mu(E)^\frac{\theta s}{n} \mu(B(x, C_1d(x)))^\frac{(1-\theta)s}{n} \, d(x)^{a(1-\frac{s}{n})} \, dx
\end{align*}
for any $0\le \theta\le 1$, where in the second step we have used a well-known result (see, e.g., \cite[formula (7.2.6)]{Jo}).

 Now, if we set $\theta=\frac{n}{sq'}$ and use that $\mu(B(x, C_1d(x)))\le d(x)^{n+a}$ we have
 \begin{align*}
I_1 &\le C \mu(E)^\frac{1}{q'}  \int_\O \frac{h(x)}{t} d(x)^b \, dx.
\end{align*}
So we only need to check that  $0\le \frac{n}{sq'} \le 1$, that holds because $1\le q\le \frac{n}{n-s}$.

 We proceed now to $I_2$. Using that $|x-y|\ge \frac{d(x)}{2}$ we have
 \begin{align*}
I_2 &\le C \int_\O \frac{h(x)}{t} \int_{E \cap B(x, C_1d(x))} \frac{d(y)^a}{d(x)^{n-s}} \, dy  \, dx\\
&= C \int_\O \frac{h(x) d(x)^{s-n}}{t}  \mu(E \cap B(x, C_1d(x)))   \, dx\\
&\le C \int_\O \frac{h(x) d(x)^{s-n}}{t}   \mu(E)^\theta \mu(B(x, C_1d(x)))^{(1-\theta)}  \, dx\\
&\le C \mu(E)^\frac{1}{q'} \int_\O \frac{h(x)}{t} d(x)^b \, dx
\end{align*}
where this time we have chosen $\theta=\frac{1}{q'}$, that clearly satisfies $0\le \theta\le 1$.

Finally, we arrive at
 $$
 \mu(E)^\frac{1}{q} t \le C \int_\O h(x) d(x)^b \, dx
 $$
 and this in turn implies, by Lemma \ref{weak-strong} with $d\nu = d(x)^\frac{b}{2} dx$, the strong inequality
 $$
\inf_{c\in\R} \|f-c\|_{L^q(\O, d^a)}\le C \int_\O \int_{|x-y|\le d(x)/2} \frac{|f(y)-f(x)|}{|y-x|^{n+s}} \, \delta(x,y)^b \, dx \, dy
 $$
 where we have used  that $\frac{d(x)}{2}\le d(y)\le \frac32 d(x)$ to replace each of these distances by $C \delta(x,y)$.
  \end{proof}

\section{The case of $\beta$-John domains}

In this section we obtain a representation analogous to that of Proposition \ref{representacion fraccionaria} in the case of $\beta$-John domains, for $\beta>1$. Observe that, although the estimate also holds for $\beta=1$, it is not only more complicated but also slightly worse than that of  Proposition \ref{representacion fraccionaria} in the case $p>1$, since it includes the restriction $b<sp-p+1-n+\frac{p-1}{\beta}+\frac{p(n+a)}{q\beta}$.
For this reason the weighted inequalities inherit this restriction, although we believe they should hold also in the case of equality. An example at the end of the Section shows that our results are sharp except at this endpoint.

To simplify calculations, throughout this section we assume, as we may by dilating $\Omega$, that  $d(0)=15$.

\begin{prop}
\label{representacion fraccionaria en beta John}
Given $s\in(0,1), a\ge 0$ and $f\in C^\infty(\O)$ we have
$$
|f(y) - \bar f|
\le C \int_{|x-y|< C_1 d(x)}\frac{h(x)}{|x-y|^{n-s}} \, dx
+ C \Big(\int_{|x-y|<C_2 d(x)^\frac{1}{\beta}} h(x)^p \, d(x)^{b-\frac{(n+a)p}{\beta q}} dx \Big)^\frac{1}{p}
$$
where $b<sp-p+1-n+\frac{p-1}{\beta}+\frac{p(n+a)}{q\beta}$ if $p>1$, and $b=s-n+\frac{(n+a)}{\beta q}$  if $p=1$, $\bar f$ is a constant and
$$
h(x)
:=\left(\int_{|x-w|\le d(x)/2} \frac{|f(w)-f(x)|^p}{|w-x|^{n+sp}} \, dw\right)^\frac{1}{p}
$$
for $y\in\O$, $h\equiv 0$ outside $\O$, and $C$, $C_1$ and $C_2$ are positive
constants depending only on $n$ and $\O$.
\end{prop}
\begin{proof}
It suffices to prove the result for $b$ close enough to the endpoint value. We consider, as before, $\varphi \in C_0^1(B(0, \lambda/2))$ such that $\int \varphi=1$, and set
$$
u(x,t)=(f* \varphi_t)(x)
$$
Then, following \cite{DMRT}, we define
$$
\tau(y)=\inf\{t : \gamma(t,y)\cap B(y, d(y)/2)=\emptyset \}
$$
and
$$
\rho(t,y)=\begin{cases}
\xi |y-\gamma(t,y)| &\mbox{ if } t\le \tau(y)\\
\frac{1}{15} d(\gamma(t,y)) &\mbox{ if } t> \tau(y)\\
\end{cases}
$$

where $\xi$ is chosen so that $\rho(\cdot,y)$ is a continuous function, that is,
$$
\xi=\frac{2}{15} \frac{d(\gamma(\tau(y),y))}{d(y)}.
$$

Notice that $\frac{1}{15}\le \xi\le \frac15$ since
$$
d(\gamma(\tau(y),y))\le |\gamma(\tau(y),y)-y|+d(y)=\frac{d(y)}{2}+d(y)=\frac32 d(y)
$$
and
$$
d(y)\le |y-\gamma(\tau(y),y)|+d(\gamma(\tau(y),y)) \Rightarrow d(y)\le 2d(\gamma(\tau(y),y)).
$$
Also, remark that $\rho(0,y)=0$ and $\rho(1,y)=1$ and that  $\gamma(t,y)+\rho(t,y)z \in \Omega$ for every $t\in [0,1]$ and $z\in B(0, \lambda/2)$ (see \cite{DMRT} for details).
Hence, if we define
$$
\eta(t)=u(\gamma(t,y)+\rho(t,y)z, \rho(t,y))
$$
we have that

\begin{align*}
f(y)-(f*\varphi)(z)&=u(y,0)-u(z,1) =\eta(0)-\eta(1) =-\int_0^1 \eta'(t) \, dt\\
&=-\int_0^1 \nabla u(\gamma(t,y)+\rho(t,y)z, \rho(t,y)) \cdot (\dot\gamma(t,y)+\dot\rho(t,y)z) \\
& \quad - \frac{\partial u}{\partial t}(\gamma(t,y)+\rho(t,y)z, \rho(t,y)) \cdot \dot\rho(t,y) \, dt
\end{align*}

Then, if $\bar f= \int (f*\varphi)(z) \varphi(z) \, dz$, we have

\begin{align*}
f(y)-\bar f &= -\int_{\R^n} \int_0^1 \nabla u(\gamma+\rho z, \rho)\cdot (\dot \gamma + \dot \rho z) \varphi(z) \, dt \, dz\\
&\quad-\int_{\R^n} \int_0^1 \frac{\partial u}{\partial t} (\gamma+\rho z, \rho) \dot\rho \varphi(z) \, dt \, dz\\
&=- (I + II)
\end{align*}

To estimate $I$, we make the change of variables
$$
\gamma(t,y)+\rho(t,y) z=x  \quad , \quad dz=\frac{dx}{\rho^n(t,y)}
$$
and use the definition of $u$ and the support of $\varphi$ to arrive at
\begin{align*}
I &= \int_{\R^n} \int_0^1 \int_{|x-w|<\lambda \rho/2} f(w) \frac{1}{\rho^{n+1}}
\nabla\varphi\Big(\frac{x-w}{\rho}\Big) \, dw \Big(\dot\gamma+
\Big(\frac{x-\gamma}{\rho}\Big)\dot\rho\Big) \varphi\Big(\frac{x-\gamma}{\rho}\Big) \frac{dt}{\rho^n} \, dx\\
\end{align*}

Now we use that $\int \frac{1}{\rho^{n+1}} \nabla\varphi \Big(\frac{x-w}{\rho}\Big) dw =0$
(to subtract $f(x)$ in the integral with respect to $w$),  that the integrand vanishes unless $|x-w|<\frac{\lambda}{2} \rho(t,y)$, that $|x-\gamma(t,y)|\le \frac{\lambda}{2}\rho(t,y)$ (both because of the support of $\varphi$) and that $\dot \rho(t,y)= \nabla d(\gamma(t,y)) \cdot \dot\gamma(t,y)$ (whence, $|\dot \rho|\le |\dot \gamma|$), to write

\begin{align*}
I &\le \int_{\R^n} \int_0^1
\int_{|x-w|<\lambda \rho/2} (f(w)-f(x)) \frac{1}{\rho^{n+1}} \nabla\varphi\Big(\frac{x-w}{\rho}\Big) \, dw \Big(\dot\gamma+ \Big(\frac{x-\gamma}{\rho}\Big) \dot\rho\Big) \varphi\Big(\frac{x-\gamma}{\rho}\Big) \frac{dt}{\rho^n} \, dx\\
 &\le C \int_{\R^n} \int_0^1 \int_{|x-w|<\lambda \rho/2} \frac{|f(w)-f(x)|}{|x-w|^{\frac{n}{p}+s}} \frac{1}{\rho^\frac{n}{p'}} \Big|\nabla\varphi\Big(\frac{x-w}{\rho}\Big)\Big| \, dw |\dot \gamma| \Big|\varphi \Big(\frac{x-\gamma}{\rho} \Big) \Big| \,  \frac{dt}{\rho^{n+1-s}} dx\\
\end{align*}

Now, we claim that $\lambda \rho(t,y)< d(x)$. Indeed, when $t\in [0, \tau(y)]$, we have that
$$
\rho(t,y)=\xi |y-\gamma(t,y)|\le \frac{\xi}{2} d(y)
$$
because $\gamma$ is inside $B(y, d(y)/2)$. But,
\begin{align}
\label{cota-cerca}
d(y)&\le |x-y|+d(x) \\
&\le |x-\gamma(t,y)|+|\gamma(t,y)-y|+d(x) \nonumber\\
&\le \frac{\lambda}{2}\rho(t,y)+\frac{1}{\xi} \rho(t,y)+d(x) \nonumber\\
&= \Big(\frac{\lambda}{2}+\frac{1}{\xi}\Big) \rho(t,y)+d(x).\nonumber
\end{align}
So,
$$
\lambda \rho(t,y)\le \frac{\lambda}{2-\xi\lambda}d(x) < d(x).
$$

On the other hand, if $t\in [\tau(y),1]$, we have that
$$
\rho(t,y) = \frac{1}{15} d(\gamma(t,y)) \le \frac{1}{15} |x-\gamma(t,y)| + \frac{1}{15} d(x)  \le \frac{\lambda}{30} \rho(t,y) + \frac{1}{15} d(x)
$$
so that
\begin{equation}
\label{cota-lejos}
\lambda \rho(t,y)\le \frac{2\lambda}{30-\lambda} d(x)  < d(x).
\end{equation}
Hence, using these bounds and H\"older's inequality (with the usual modification if $p=1$) we obtain
\begin{align*}
I &\le C \int_{\R^n} \int_0^1 \Big( \int_{|x-w|<d(x)/2} \frac{|f(w)-f(x)|^p}{|x-w|^{n+sp}} dw \Big)^\frac{1}{p} \\
&\quad \cdot \Big(\int_{\R^n} \Big|\nabla \varphi \Big(\frac{x-w}{\rho}\Big) \Big|^{p'} \frac{1}{\rho^n} dw \Big)^\frac{1}{p'} |\dot \gamma| \Big|\varphi\Big( \frac{x-\gamma}{\rho}\Big)\Big| \frac{dt}{\rho^{n+1-s}} dx\\
&\le C\int_{\R^n} \Big( \int_{|x-w|<d(x)/2} \frac{|f(w)-f(x)|^p}{|x-w|^{n+sp}} dw \Big)^\frac{1}{p} \int_{0}^1  |\dot \gamma| \Big|\varphi\Big( \frac{x-\gamma}{\rho}\Big)\Big| \frac{dt}{\rho^{n+1-s}} dx\\
&= C\int_{\R^n} h(x) \int_{0}^1  |\dot \gamma| \Big|\varphi\Big( \frac{x-\gamma}{\rho}\Big)\Big| \frac{dt}{\rho^{n+1-s}} dx\\
&= C \Big( \int_{\R^n} h(x) \int_{0}^{\tau(y)} \dots dt dx  +  \int_{\R^n} h(x) \int_{\tau(y)}^1 \dots \, dt dx \Big)\\
&= C (I_a + I_b)
\end{align*}

For $I_a$, notice that proceeding as in \eqref{cota-cerca} we have that
$$
|x-y|\le \Big( \frac{\lambda}{2} +\frac{1}{\xi}\Big) \rho(t,y) <  \Big( \frac12 +\frac{1}{\lambda \xi}\Big) d(x)<C_1 d(x).
$$

Thus, we can write

\begin{align*}
I_a & \le C \int_{|x-y|< C_1 d(x)} h(x) \int_{0}^{\tau(y)}  |\dot \gamma| \Big|\varphi\Big( \frac{x-\gamma}{\rho}\Big)\Big| \frac{dt}{|x-y|^{n+1-s}} dx\\
\end{align*}

Now, the integral vanishes unless
 $$
 |x-\gamma(t,y)|\le \rho(t,y) \frac{\lambda}2 \le \rho(t,y)
 =\xi|y-\gamma(t,y)|
 \le\xi|x-\gamma(t,y)|+\xi|x-y|
 $$
 which implies
 $$
 |x-\gamma(t,y)|\le \frac{\xi}{1-\xi} |x-y|\le \frac14 |x-y|,
 $$
so we can bound
 \begin{align*}
I_a &\le C \int_{|x-y|< C_1 d(x)} h(x) \int_0^{\tau(y)}   \chi_{ |x-\gamma(t,y)|\le  \frac14 |x-y|} |\dot\gamma(t,y)|  \frac{dt}{|x-y|^{n+1-s}} \, dx\\
&\le C \int_{|x-y|< C_1 d(x)}  h(x) \; \ell(\gamma(y) \cap B(x, |x-y|/4))   \frac{1}{|x-y|^{n+1-s}} \, dx\\
&\le C \int_{|x-y|< C_1 d(x)}     \frac{h(x)}{|x-y|^{n-s}} \, dx
 \end{align*}
where in the last step we have used property \eqref{prop5} of $\beta$-John domains.

To bound $I_b$, observe that for $t\in [\tau(y), 1]$ we have
\begin{equation}
d(x)\le d(\gamma(t,y)) + |x-\gamma(t,y)|\le 15 \rho(t,y) + \frac{\lambda}{2} \rho(t,y),
\label{cota-d-rho}
\end{equation}
and, because we are in a $\beta$-John domain, and by \eqref{cota-lejos} we have that
\begin{equation}
\label{cota en beta John}
|x-y|\le |x-\gamma(t,y)|+|\gamma(t,y)-y|<\frac{\lambda}{2}\rho(t,y)
+ C|\dot\gamma(t,y)| t< C_2d(x)^\frac{1}{\beta}
\end{equation}
so that, if $p=1$,
\begin{align*}
I_b&= \int_{\R^n} h(x) \int_{\tau(y)}^1  |\dot \gamma| \Big|\varphi\Big( \frac{x-\gamma}{\rho}\Big)\Big|
\frac{dt}{\rho^{n+1-s}} dx\\
&\le C \int_{\R^n} h(x) \int_{\tau(y)}^1 |\dot \gamma| \Big| \varphi\Big(\frac{x-\gamma}{\rho}\Big)\Big|
\frac{dt}{d(x)^{n+1-s}} dx\\
&\le C \int_{|x-y|<C_2 d(x)^\frac{1}{\beta}} h(x) \int_{\tau(y)}^1
\chi_{ |x-\gamma(t,y)|\le  \frac{d(x)}{2}}|\dot \gamma| \frac{dt}{d(x)^{n+1-s}} dx\\
&\le C \int_{|x-y|<C_2 d(x)^\frac{1}{\beta}} h(x)  \; \ell(\gamma(y) \cap B(x,d(x)/2)){d(x)^{s-n-1}} dx\\
&\le C\int_{|x-y|<C_2 d(x)^\frac{1}{\beta}} h(x) \, d(x)^{s-n} dx\\
\end{align*}
where we have used \eqref{cota-lejos} and property \eqref{prop5}.

If  $p>1$, to bound $I_b$, by H\"older's inequality and property \eqref{prop5} we have
\begin{align*}
I_b&= \int_{\R^n} h(x) \int_{\tau(y)}^1  |\dot \gamma|
\Big|\varphi\Big( \frac{x-\gamma}{\rho}\Big)\Big| \frac{dt}{\rho^{n+1-s}} dx\\
&\le C  \int_{|x-y|<C_2 d(x)^\frac{1}{\beta}}  h(x) \Big(  \int_{\tau(y)}^1 \chi_{|x-\gamma|<\frac{d(x)}{2}} |\dot \gamma| \, dt \Big)^\frac{1}{p}
 \Big(\int_{\tau(y)}^1  \Big| \varphi\Big(\frac{x-\gamma}{\rho}\Big)\Big|^{p'} \frac{1}{\rho^{(n+1-s) p'}}  |\dot \gamma| \, dt  \Big)^\frac{1}{p'} dx\\
 &\le C  \int_{|x-y|<C_2 d(x)^\frac{1}{\beta}}  h(x) d(x)^\frac{1}{p}
 \Big(\int_{\tau(y)}^1  \Big| \varphi\Big(\frac{x-\gamma}{\rho}\Big)\Big|^{p'} \frac{1}{\rho^{(n+1-s) p'}}  |\dot \gamma| \, dt  \Big)^\frac{1}{p'} dx\\
 \end{align*}
 Therefore, since $\rho\sim d(x)$ by \eqref{cota-lejos} and \eqref{cota-d-rho}, using H\"older's inequality again we arrive at
 \begin{align*}{}
 I_b &
 \le C  \int_{|x-y|<C_2 d(x)^\frac{1}{\beta}}  h(x) d(x)^{(b-\frac{(n+a)p}{\beta q})\frac{1}{p}}
 \Big(\int_{\tau(y)}^1  \Big| \varphi\Big(\frac{x-\gamma}{\rho}\Big)\Big|^{p'} \frac{1}{\rho^{(n+1-s) p'
 + (b-\frac{(n+a)p}{\beta q}-1)\frac{p'}{p} }} \, dt  \Big)^\frac{1}{p'} dx\\
&\le C  \Big( \int_{|x-y|<C_2 d(x)^\frac{1}{\beta}}
h(x)^p d(x)^{b-\frac{(n+a)p}{\beta q}} \, dx \Big)^\frac{1}{p}  \\
 &\quad\cdot \Big(\int_{\tau(y)}^1 \int_{\R^n}
 \Big| \varphi\Big(\frac{x-\gamma}{\rho}\Big)\Big|^{p'} \frac{1}{\rho^n} \, dx \,
 \frac{1}{\rho^{-n+(n+1-s) p'+ (b-\frac{(n+a)p}{\beta q}-1)\frac{p'}{p} }} \, dt  \Big)^\frac{1}{p'}\\
\end{align*}
and finally, by property \eqref{prop2},
\begin{align*}
 I_b  &\le C  \Big( \int_{|x-y|<C_2 d(x)^\frac{1}{\beta}}  h(x)^p d(x)^{b-\frac{(n+a)p}{\beta q}} \, dx \Big)^\frac{1}{p}
 \Big(\int_{0}^1  \frac{1}{t^{\beta[-n+(n+1-s) p'+ (b-\frac{(n+a)p}{\beta q}-1)\frac{p'}{p} ]}} \, dt  \Big)^\frac{1}{p'}\\
 &\le C  \Big( \int_{|x-y|<C_2 d(x)^\frac{1}{\beta}}  h(x)^p d(x)^{b-\frac{(n+a)p}{\beta q}} \, dx \Big)^\frac{1}{p}
\end{align*}
provided $0\le\beta[-n+(n+1-s) p'+ (b-\frac{(n+a)p}{\beta q}-1)\frac{p'}{p}]<1$ which holds for
$b<ps-p+1-n+\frac{p-1}{\beta}+\frac{(n+a)p}{\beta q}$ and sufficiently close to that number.

To estimate $II$ we proceed in a similar way. Indeed, since $\int\varphi_t(x)dx=1$ for all
$t$, we have $\int\frac{\partial\varphi_t}{\partial t}(x)dx=0$. Moreover, recalling that
$$
\frac{\partial\varphi_t}{\partial t}(x)
=\frac1{t^{n+1}}\psi\Big(\frac{x}{t}\Big)
$$
with $\psi :=-n\varphi-x\cdot\nabla\varphi$. Therefore, repeating the arguments
that we used to bound $I$ we obtain,

\begin{align*}
II &= \int_{\R^n} \int_0^1
\int_{|x-w|<\lambda \rho/2} (f(w)-f(x)) \frac{1}{\rho^{n+1}} \psi\Big(\frac{x-w}{\rho}\Big) \, dw
\dot\rho \varphi\Big(\frac{x-\gamma}{\rho}\Big)
\frac{dt}{\rho^n} \, dx
\end{align*}
which can be bounded analogously. This completes the proof.
\end{proof}

Using the above representation we obtain the improved inequalities in the case of $\beta$-John domains:

\begin{theorem}
\label{teorema en beta John}
Let $\Omega\subset\R^n$ be a $\beta$-John domain,  $1< p\le q<\infty$, $a\ge 0$,
 $b<\frac{(n+a)p}{q\beta}+\frac{p-1}{\beta}+sp-p+1-n$ and, additionally,
 $q\le \frac{n-p}{n-sp}$ when $p<\frac{n}{s}$. Given $s\in(0,1)$ and $f\in C^\infty(\O)$ we have
$$
\inf_{c\in\R}\|f(y) - c\|_{L^q(\O, d^a)}
\le C \left\{
\int_\O\int_{|x-z|\le d(x)/2}\frac{|f(z)-f(x)|^p}{|z-x|^{n+sp}} \, \delta(x,z)^{b} dz dx
\right\}^{1/p}
$$
where $\delta(x,z):=\min\{d(x),d(z)\}$.

\end{theorem}

\begin{proof}

By Proposition \ref{representacion fraccionaria en beta John} we have
\begin{align*}
|f(y) - \bar f| &
\le C \int_{|x-y|< C_1 d(x)}\frac{h(x)}{|x-y|^{n-s}} \, dx
+ C \Big(\int_{|x-y|<C_2 d(x)^\frac{1}{\beta}} h(x)^p \, d(x)^{b-\frac{(n+a)p}{\beta q}} dx \Big)^\frac{1}{p} \\
&= A +B
\end{align*}

Observe that $\|A\|_{L^q(\Omega, d^a)}$ can be bounded as in Theorem  \ref{teorema en John},
 so it suffices to restrict ourselves to $\|B\|_{L^q(\Omega, d^a)}^p=\|B^p\|_{L^\frac{q}{p}(\Omega, d^a)}$. We have, using Minkowski's integral inequality,

 \begin{align*}
 \|B^p\|_{L^\frac{q}{p}(\Omega, d^a)} & = C \left[ \int_\Omega \Big( \int_{|x-y|<C_2 d(x)^\frac{1}{\beta}} h(x)^p d(x)^{b-\frac{(n+a)p}{\beta q}} dx  \Big)^\frac{q}{p} d(y)^a \, dy \right]^\frac{p}{q}\\
 & \le C  \int_\Omega h(x)^p d(x)^{b-\frac{(n+a)p}{\beta q}}  \Big( \int_{|x-y|<C_2 d(x)^\frac{1}{\beta}} d(y)^a \, dy  \Big)^\frac{p}{q}  \, dx \\
 & \le C  \int_\Omega h(x)^p d(x)^{b-\frac{(n+a)p}{\beta q}}  \Big( d(x)^\frac{a}{\beta} \int_{|x-y|<C_2 d(x)^\frac{1}{\beta}}  \, dy  \Big)^\frac{p}{q}  \, dx \\
 &= C  \int_\Omega h(x)^p d(x)^b  \, dx
 \end{align*}
where again we have used that $d(y)\le |x-y| +d(x)\le Cd(x)^\frac{1}{\beta}$
and that $a\ge 0$. Therefore,
 $$
 \|B\|_{L^q(\Omega, d^a)} \le C \Big( \int_\Omega h(x)^p d(x)^b  \, dx \Big)^\frac{1}{p}.
 $$
 This concludes the proof.
\end{proof}

As discussed before, in the case $p=1$ we recover the endpoint value for $b$ and can prove the following:

\begin{theorem}
\label{teo 1 en beta John}
Let $\Omega\subset\R^n$ be a $\beta$-John domain,  $1\le q\le \frac{n-1}{n-s}$, $a\ge 0$ and
 $b\le\frac{(n+a)}{q\beta}+s-n$. Given $s\in(0,1)$ and $f\in C^\infty(\O)$ we have
$$
\inf_{c\in\R}\|f(y) - c\|_{L^q(\O, d^a)}
\le C
\int_\O\int_{|x-z|\le d(x)/2}\frac{|f(z)-f(x)|}{|z-x|^{n+s}} \, \delta(x,z)^{b} dz dx
$$
where $\delta(x,z):=\min\{d(x),d(z)\}$.
\end{theorem}

\begin{proof}

Clearly, it suffices to prove the result for  $b=\frac{n+a}{q\beta}+s-n$. By Proposition \ref{representacion fraccionaria en beta John}, we have
\begin{align*}
|f(y) - \bar f| &
\le C \int_{|x-y|< C_1 d(x)}\frac{h(x)}{|x-y|^{n-s}} \, dx
+ C \int_{|x-y|<C_2 d(x)^\frac{1}{\beta}} h(x) \, d(x)^{s-n} dx  \\
&= A +B
\end{align*}

Observe that $\|A\|_{L^1(\Omega, d^a)}$ and $\|A\|_{L^{q,\infty}(\Omega, d^a)}$ (for $q>1$) can be bounded as in Theorem  \ref{caso 1 en John}.
So we must consider $\|B\|_{L^q(\Omega, d^a)}$ ($q\ge 1$). By Minkowski's integral inequality, we have
\begin{align*}
\|B\|_{L^q(\Omega, d^a)} &\le C \int_\Omega h(x) \, d(x)^{s-n} \Big( \int_{|x-y|<C_2 d(x)^\frac{1}{\beta}}   d(y)^a \, dy \Big)^\frac{1}{q} \, dx\\
&\le C \int_\Omega h(x) \, d(x)^{s-n} \Big( d(x)^\frac{a}{\beta} \int_{|x-y|<C_2 d(x)^\frac{1}{\beta}}   \, dy  \Big)^\frac{1}{q} \, dx\\
&= C \int_\Omega h(x) \, d(x)^{s-n+\frac{a+n}{q\beta}} \, dx
\end{align*}
 where in the second line we have used that $d(y)\le |x-y| +d(x)\le Cd(x)^\frac{1}{\beta}$ and
 that $a\ge 0$.

The result then follows immediately for $q=1$ and using the ``weak implies strong" technique as in  Theorem  \ref{caso 1 en John} for $q>1$.
\end{proof}

To analyze the optimality of our estimates (in terms of the upper bound on $q$) we consider the following
``rooms and corridors" domain introduced in  \cite{HK}  (see the discussion after Corollary 5).Therefore, we will be somewhat sketchy.

\begin{theorem}
\label{mushroom}
Let $s\in(0,1), a\ge 0$ and $1\le p\le q<\infty$. There exist a $\beta$-John domain $\O\subset\R^n$
and $f\in C^\infty(\O)$ such that
$$
\inf_{c\in\R}\|f(y) - c\|_{L^q(\O, d^a)}
\le C \left\{
\int_\O\int_{|x-z|\le d(x)/2}\frac{|f(z)-f(x)|^p}{|z-x|^{n+sp}} \, \delta(x,z)^{b} dz dx
\right\}^{1/p}
$$
cannot hold unless $q\le\frac{(n+a)p}{1-p+\beta(b+n-1+p-sp)}$.
\end{theorem}

\begin{proof}
Assume, for simplicity, that $a=b=0$.

Following \cite{HK}, we define a `mushroom' $F$ of size $r$ as the union of a cylinder of height
$r$ and radius $r^\beta$ (called the `stem' and denoted by $\mathcal{P}$) with a ball of radius $r$
(called the `cap' and denoted by $\mathcal{C}$), so that they create a mushroom-like shape.
The domain $\O$ considered consists of a cube $Q$ and an infinite sequence of disjoint mushrooms
$F_1, F_2, \dots$   on one side of the cube (called the `top'). The stems of $F_1, F_2, \dots$
are perpendicular to the top and of decreasing size $r_i\to 0$.
This domain $\O$ can easily be seen to be a $\beta$-John domain.

Now, we let $f_i$ be the piecewise linear function on $\O$ such that $f_i=0$ outside
$F_i$, $f_i=1$ on the cap and $f_i$ is linear on the stem. We may also assume that
$\bar f_i=0$ for every $i$. Hence, $\|f_i-\bar f_i\|_{L^q(\O)} \ge c_n r_i^\frac{n}{q}$.

To bound $(\int_\O \int_{|x-z|<d(x)/2} \frac{|f_i(x)-f_i(z)|^p}{|x-z|^{n+sp}} dz dx)^{1/p}$ observe that:
\begin{itemize}
\item if $x \in \mathcal{C}_i, z\in Q$, then the integral vanishes, since  $d(x)\le r_i$ and $|x-z|>r_i$
\item if $x, z\in Q$ or $x, z\in \mathcal{C}_i$, then $|f_i(x)-f_i(z)|=0$
\item in all remaining cases, $|f_i(x)-f_i(z)|^p\sim r_i^{-p} |x_n-z_n|^p$
and $\frac12 d(x) \le d(z) \le \frac32 d(x)$, so  $d(x) \sim d(z) \le  r_i^\beta$
\end{itemize}
Then,
\begin{align*}
\int_\O \int_{|x-z|<d(x)/2} \frac{|f_i(x)-f_i(z)|^p}{|x-z|^{n+sp}} dz dx &=
\int_{Q}  \cdots \, dx + \int_{\mathcal{P}_i}    \cdots \,dx+  \int_{\mathcal{C}_i}  \cdots \,dx\\
&= I_1 + I_2 + I_3
\end{align*}
We have
\begin{align*}
I_1 &\le \int_{Q} \int_{|x-z|<d(x)/2} \frac{1}{r_i^p} \frac{|x_n -z_n|^p}{|x-z|^{n+sp}} \, dz \, dx \\
&\le C \int_{\cup_{w\in \mathcal{P}_i} B(w, r_i^\beta) \cap Q} \frac{1}{r_i^p} d(x)^{-sp+p} \, dx \\
&\le C |\cup_{w\in \mathcal{P}_i} B(w, r_i^\beta) \cap Q| \, r_i^{-p+\beta(-sp+p)}\\
&\le C  r_i^{-p+\beta(n-sp+p)}\\
\end{align*}

\begin{align*}
I_2 &\le \int_{\mathcal{P}_i} \int_{|x-z|<d(x)/2} \frac{1}{r_i^p} \frac{|x_n -z_n|^p}{|x-z|^{n+sp}} \, dz \, dx \\
&\le C \int_{\mathcal{P}_i} \frac{1}{r_i^p} d(x)^{-sp+p} \, dx \\
&\le C |\mathcal{P}_i| r_i^{-p+\beta(-sp+p)}\\
&= C r_i^{1-p+\beta(-sp+p+n-1)}
\end{align*}

\begin{align*}
I_3 &\le \int_{\mathcal{C}_i} \int_{|x-z|<d(x)/2} \frac{1}{r_i^p} \frac{|x_n -z_n|^p}{|x-z|^{n+sp}} \, dz \, dx \\
&\le C \int_{\cup_{w\in \mathcal{P}_i} B(w, r_i^\beta) \cap \mathcal{C}_i} \frac{1}{r_i^p} d(x)^{-sp+p} \, dx \\
&\le C |\cup_{w\in \mathcal{P}_i} B(w, r_i^\beta) \cap \mathcal{C}_i| r_i^{-p+\beta(-sp+p)}\\
&\le C  r_i^{-p+\beta(n-sp+p)}\\
\end{align*}

Then, there must hold $r_i^\frac{n}{q}\le C (r_i^{1-p+\beta(-sp+p+n-1)})^\frac{1}{p}$ which,
for sufficiently small $r_i$, can only hold if $q\le \frac{np}{1-p+\beta(n-1+p-sp)}$, as we wanted to prove.

It is easy to see that the same example can be used to prove the optimality in the general case (with $a, b$ not necessarily $0$).
\end{proof}

\section{The case of H\"older-$\alpha$ domains}

Roughly speaking, a H\"older-$\alpha$ domain is given locally by the hypograph, in an appropriate
orthogonal system, of
a H\"older-$\alpha$ function (a typical example being a cuspidal domain).
For a precise definition we refer to \cite[Section 5.2]{DMRT}.

These domains are a particular case of $\beta$-John domains with $\beta=1/\alpha$. However,
they are known to have better embedding properties, as they cannot contains ``rooms and corridors" like
 general $\beta$-John domains (see, e.g. \cite{BS} and \cite[Example 2.4]{KM}).
 Therefore, it is natural that our result of Theorem \ref{teorema en beta John} can be improved in this case.

We obtain the following result, which is an improvement of Theorem  \ref{teorema en beta John} when $\O$
is H\"older-$\alpha$, $a=0$, and $p>1$, and we prove its optimality. We believe that these restrictions are only technical.

\begin{theorem}
\label{teorema en Holder alpha}
Let $\Omega\subset\R^n$ be a H\"older-$\alpha$ domain, $1< p\le q<\infty$, $b\le p(s-n)+p(n-1+\alpha)(1+\frac{1}{q}-\frac{1}{p})$,
and, additionally,
 $q\le \frac{n-p}{n-sp}$ when $p<\frac{n}{s}$. Given $s\in(0,1)$ and $f\in C^\infty(\O)$ we have$$
\inf_{c\in\R} \|f(y) - \bar f\|_{L^q(\O)}\le C \left\{\int_\O\int_{|x-z|\le d(x)/2}\frac{|f(z)-f(x)|^p}{|z-x|^{n+sp}}\delta(x,z)^b dz dx
\right\}^{1/p}
$$
where $\delta(x,z):=\min\{d(x),d(z)\}$.
\end{theorem}

\begin{proof}

Clearly, given $p$ and $q$, it suffices to prove the claim for $b= p(s-n)+p(n-1+\alpha)(1+\frac{1}{q}-\frac{1}{p})$.

Recall that by Proposition \ref{representacion fraccionaria en beta John} we could use
\begin{equation}
\label{ra}
|f(y) - \bar f|
\le C \int_{|x-y|< C_1 d(x)}\frac{h(x)}{|x-y|^{n-s}} \, dx
+ C \int_{|x-y|<C_2 d(x)^\alpha} h(x) \, d(x)^{s-n} dx.
\end{equation}

The key point is to improve this estimate,  observing that \eqref{cota en beta John} can be improved if $\Omega$ is H\"older-$\alpha$. According to the definition given in \cite{DMRT}, $\O= \cup_{j=0}^N O_j$ with $\bar O_0 \subset \O$ and
each $\O \cap O_j$ ($1\le j\le N$) given by the hypograph of a H\"older-$\alpha$ function
in an appropriate coordinate system.
It is clear that it is enough to obtain the desired estimate for each $O_j$ with fixed $j\ge 1$.
To simplify notation, assume that in a given $O_j$ the appropriate coordinate system is the usual one,
$x=(x',x_n)$, $x'\in\R^{n-1}$, $x_n\in\R$.

Now, by \eqref{cota-lejos}, $|x-\gamma(t,y)|\le d(x)/2$. But, it was proved in \cite[Section 5.2]{DMRT}
that this inequality implies that
$$
|x'-y'|\le C d(x) \qquad , \qquad
|x_n-y_n|\le C d(x)^\alpha
$$
and then we can replace the second integral in \eqref{ra} by
\begin{align*}
\int_{|x'-y'|\le Cd(x), |x_n-y_n|\le Cd(x)^{\alpha} } h(x)
d(x)^{s-n} \, dx\\
\end{align*}

Consequently, proceeding by duality,  for $\|g\|_{L^{q'}(\O)}=1$, we have
$$
\int_{O_j} |f(y)-\bar f| g(y) dy
\le A + B
$$
where $A$ is as in \eqref{dualidad} (with $a=0$), and can be bounded similarly. To estimate $B$ we consider separately the cases $p=q$ and $p<q$.

If $p=q$ we write
\begin{align*}
B &\le C \int_{\O} \frac{1}{d(x)^{n-1+\alpha}} \int_{|x'-y'|\le Cd(x),
|x_n-y_n|\le Cd(x)^{\alpha} } |g(y)|
\, dy h(x) d(x)^{s-1+\alpha} \, dx\\
&\le C \int_{\O} M^S\,g(x) h(x) d(x)^{s-1+\alpha} \, dx \\
&\le C \|M^S\,g\|_{L^{p'}(\O)} \|h\,d^{s-1+\alpha}\|_{L^{p}(\O)}\\
\end{align*}
where $M^S$ is the strong maximal function, i.e., the maximal function over the basis of rectangles with sides parallel to the axes, which is known to be bounded in $L^p$ for $1<p\le\infty$.
Then, the proof concludes as that of Theorem \ref{teorema en beta John} and adding over $j$.

If $p<q$,  taking $\eta=\frac{n}{p}-\frac{n}{q}$  (notice that $0<\eta<n$) we have,
\begin{align*}
B &\le C \int_{\O} \frac{1}{d(x)^{(n-1+\alpha)(1-\frac{\eta}{n})}} \int_{|x'-y'|\le Cd(x), |x_n-y_n|\le Cd(x)^{\alpha} } |g(y)| \, dy h(x) d(x)^{s-n+(n-1+\alpha)(1-\frac{\eta}{n})} \, dx\\
&\le C \int_{\O} M^S_{\eta}g(x) h(x) d(x)^{s-n+(n-1+\alpha)(1-\frac{\eta}{n})} \, dx \\
&\le C \|M^S_{\eta}g \|_{L^{p'}(\O)} \|h\, d^\frac{b}{p}\|_{L^{p}(\O)} \\
&\le C \|g\|_{L^{q'}(\O)}  \|h\, d^\frac{b}{p}\|_{L^{p}(\O)}
\end{align*}
where we have used that $M^S_\eta g(x) := \sup_{R\ni x} \frac{1}{|R|^{1-\frac{\eta}{n}}} \int_R |f(y)| \, dy$,
 $0<\eta<n$, where $R$ belongs to the family of rectangles with sides parallel to the axes,
 and that $M^S_{\eta}: L^{q'} \to L^{p'}$ for $\frac{1}{q}=\frac{1}{p}-\frac{\eta}{n}$ and $1<p< \frac{n}{\eta}$ (see, e.g., \cite[Theorem 3.1]{Mo}). As before, the proof concludes as in  Theorem \ref{teorema en beta John} and adding over $j$.

\end{proof}

In the next Theorem we prove that the previous result is optimal with respect to the
exponent $b$. We generalize an argument given in \cite{ADL}
for the case $s=1$.

\begin{theorem} Let $1<p\le q<\infty$. There exist a H\"older-$\alpha$ domain $\O\subset\R^n$ and
$f\in C^\infty(\O)$ such that
\begin{equation}
\label{optimo}
\inf_{c\in\R} \|f - c\|_{L^q(\O)}
\le C \left\{
\int_\O\int_{|x-z|< d(x)/2}\frac{|f(z)-f(x)|^p}{|z-x|^{n+sp}}\delta(x,z)^b dz dx
\right\}^{1/p}
\end{equation}
cannot hold unless $b\le p(s-1+\alpha)+p(n-1+\alpha)(\frac{1}{q}-\frac{1}{p})$.
\end{theorem}

\begin{proof} Assume that $b>p(s-1+\alpha)+p(n-1+\alpha)(\frac{1}{q}-\frac{1}{p})$.
Using the same notation as in the previous section we write $x=(x',x_n)$.
Given $0<\alpha\le 1$ define the H\"older-$\alpha$ domain
$$
\O=\{x\in\R^n\,:\, 0<x_n<1 \, , \, |x'|<x_n^{1/\alpha}\}
$$
and $f(x)=x_n^{-\nu}$, with $\nu>0$ to be chosen.

It is not difficult to check that $d(x)\sim x_n^{1/\alpha}-|x'|$. Then,
in the subdomain $\widetilde\O\subset\O$ defined by
$$
\widetilde\O=\{x\in\R^n\,:\, 0<x_n<1 \, , \, |x'|<x_n^{1/\alpha}/2\}
$$
we clearly have $d(x)\sim x_n^{1/\alpha}$. Then,
\begin{equation}
\label{izquierda}
\inf_{c\in\R} \|f - c\|_{L^q(\widetilde\O)}
\le \|f\|^q_{L^q(\widetilde\O)}
\sim \int_0^1\int_{|x'|<x_n^{1/\alpha}/2} x_n^{-\nu q} \, dx'\, dx_n
\sim \int_0^1 x_n^{-\nu q+\frac{n-1}{\alpha}} \, dx_n
\end{equation}
On the other hand, if $|x-z|< d(x)/2$ we have $|x_n-z_n|<x_n/2$ and so $x_n\sim z_n$.
Consequently,
$$
|f(z)-f(x)|=|z_n^{-\nu}-x_n^{-\nu}|\le C x_n^{-\nu-1} |z_n-x_n|
$$
and therefore,
$$
\begin{aligned}
\int_{|x-z|< d(x)/2}\frac{|f(z)-f(x)|^p}{|z-x|^{n+sp}}\delta(x,z)^b dz
&\le C x_n^{-(\nu+1)p}d(x)^b\int_{|x-z|< d(x)/2}|z-x|^{p-n-sp} dz\\
\le C x_n^{-(\nu+1)p}d(x)^{b+(1-s)p}
&\le C x_n^{-(\nu+1)p+\frac{b+(1-s)p}{\alpha}}
\end{aligned}
$$
where in the last inequality we have used that $d(x)\le x_n^{1/\alpha}$
and $b+(1-s)p\ge 0$ (which follows from our assumptions on $b, p$ and $q$).
Then,
\begin{equation}
\label{derecha}
\begin{aligned}
\int_\O\int_{|x-z|< d(x)/2}\frac{|f(z)-f(x)|^p}{|z-x|^{n+sp}}\delta(x,z)^b \,dz\,dx
&\le C\int_\O x_n^{-(\nu+1)p+\frac{b+(1-s)p}{\alpha}}\,dx\\
&\le C\int_0^1 x_n^{-(\nu+1)p+\frac{b+(1-s)p+n-1}{\alpha}}dx
\end{aligned}
\end{equation}
Therefore, \eqref{optimo} does not hold if there exists $\nu$ such that its LHS is infinite and its RHS is finite, that is, if
$$
-\nu q+\frac{n-1}{\alpha}
\le -1
<-(\nu+1)p+\frac{b+(1-s)p+n-1}{\alpha}
$$
and then, the existence of such a $\nu$ is equivalent to
$$
b>p(s-1+\alpha)+p(n-1+\alpha)\Big(\frac{1}{q}-\frac{1}{p}\Big)
$$
as we wanted to see.
\end{proof}

\end{document}